\documentclass{article}
\usepackage{fullpage}
\usepackage{amsmath,amssymb,amsthm,graphicx,xypic,url}
\usepackage[vcentermath]{youngtab}
\usepackage{tikz}

\theoremstyle{definition}
\newtheorem{theorem}{Theorem}
\newtheorem{lemma}[theorem]{Lemma}

\newtheorem{prop}[theorem]{Proposition}
\newtheorem{example}[theorem]{Example}
\newtheorem{fact}[theorem]{Fact}

\newcommand{\wds}{[n]^*}
\newcommand{\wdsk}[1]{[n]^{#1}}

\newcommand{\defeq}{:=}
\newcommand{\ssyt}{\textsc{ssyt}}
\newcommand{\syt}{\textsc{syt}}
\newcommand{\scs}[1]{\textsc{scs}_{#1}}
\newcommand{\scssub}[2]{\textsc{scs}_{#1}(#2)}
\DeclareMathOperator{\shape}{sh}
\DeclareMathOperator{\rw}{rw}

\begin{document}

\author{
Andrew Crites \thanks{
Dept. of Mathematics,
University of Washington,
Seattle, WA 98195,
\texttt{acrites@uw.edu}}
\and
Greta Panova\thanks{
Dept. of Mathematics,
University of California, Los Angeles
Los Angeles, CA 90095
\texttt{panova@math.ucla.edu}}
\and
Gregory S. Warrington \thanks{
Dept. of Mathematics and Statistics,
University of Vermont,
Burlington, VT 05401,
\texttt{gwarring@cems.uvm.edu}.  Supported in part by National Security Agency grant H98230-09-1-0023.}}

\title{Shape and pattern containment of separable permutations}
%\footnote{Third author
    %supported in part by National Security Agency grant H98230-09-1-0023.}}

\maketitle

\begin{abstract}
  Every word has a \emph{shape} determined by its image under the
  Robinson-Schensted-Knuth correspondence.  We show that when a word
  $w$ contains a separable (i.e., 3142- and 2413-avoiding) permutation
  $\sigma$ as a pattern, the shape of $w$ contains the shape of
  $\sigma$.  As an application, we exhibit lower bounds for the
  lengths of supersequences of sets containing separable permutations.
%%   Every word has a corresponding \emph{shape} under 

%%   We study the shapes of permutations under RSK and in particular when
%%   containment as patterns implies containment of the corresponding
%%   shapes of the permutation. We show that when a subsequence $\sigma$
%%   of a word $w$ is a separable(i.e. 3142 and 2413-avoiding)
%%   permutation then the shape of $\sigma$ is contained in the shape of
%%   $w$. This result also implies an application to the study of lower
%%   bounds for shortest containing supersequences, namely, we exhibit
%%   lower bounds on the length of supersequences of certain sets of
%%   separable permutations.

%% -------------------------

%%   We study the shape of separable (3142 and 2413-avoiding)
%%   permutations under RSK.  In particular, we prove that if $\sigma$ is
%%   a separable subsequence of a word $w$, then the shape of $\sigma$ is
%%   contained in the shape of $w$ as Young diagrams.  These facts are
%%   also used to exhibit lower bounds on the lengths of words containing
%%   certain separable permutations as patterns.
\end{abstract}

The Robinson-Schensted-Knuth (RSK) correspondence associates to any
word $w$ a pair of \emph{Young tableaux}, each of equal partition
shape $\lambda = (\lambda_1,\lambda_2,\ldots)$.  We say that $w$ has
\emph{shape} $\shape(w) = \lambda$.  It is natual to expect that if
$\sigma$ is a subsequence of $w$, then $\shape(\sigma) \subseteq
\shape(w)$.  However, this is not necessarily the case: If $\sigma =
2413$ and $w = 24213$, then
\begin{equation}
  (P(\sigma),Q(\sigma)) =
  \left(\young(13,24)\,,\young(12,34)\right)\quad \text{ and }\quad
  (P(w),Q(w)) = \left(\young(123,2,4),\young(125,3,4)\right).
\end{equation}
We see that $\shape(w) = (3,1,1) \not\supseteq (2,2) =
\shape(\sigma)$.  The main theorem of this paper is that the inclusion
does hold when $\sigma$ is a \emph{separable} permutation.
Furthermore, $\sigma$ need only be contained as a pattern rather than
as an actual subsequence.

% (see
% Section~\ref{sec:notation}).

\begin{theorem}\label{thm:cont}
  If a word $w$ contains a separable permutation $\sigma$ as a
  pattern, then $\shape(w) \supseteq \shape(\sigma)$.
\end{theorem}

Our discovery of Theorem~\ref{thm:cont} was motivated by an
application involving lower bounds for \emph{shortest containing
  supersequences}.  Such supersequences arise in
bioinformatics~\cite{Ning, Rahmann} through the design of DNA
microarrays, in planning~\cite{Foulser} and in data
compression~\cite{Rodeh}.  This application to supersequences is
described in Section~\ref{sec:super}.  Section~\ref{sec:notation}
introduces the requisite notation required for the proof of
Theorem~\ref{thm:cont} appearing in Section~\ref{sec:proof}.
Section~\ref{sec:gt} discusses the relationship between Greene's
Theorem, separable permutations, and the contents of this paper.

\section{Background and setup}
\label{sec:notation}

% space of words and permutations
Let $\wds$ denote the set of finite-length words on the alphabet $[n]
\defeq \{1,2,\ldots,n\}$ and let $\wdsk{a}$ denote the subset of
length-$a$ words.  The set of permutations of length $n$ is denoted by
$S_n$ (here a subset of $\wdsk{n}$).  Permutations will be denoted by
Greek letters and written in one-line notation.  For example, the
permutation $\tau\in S_3$ defined by $\tau(1) = 3$, $\tau(2) =
1$ and $\tau(3) = 2$ is written $312$.  When referring to a
subsequence of a permutation $\tau$ we make no distinction between
the actual subsequence and the corresponding subset of elements; the
subsequence can be reconstructed by the positions in $\tau$.  The
length of a word $u$ is denoted $|u|$.

%% Given permutations $\tau\in S_n$ and $\pi\in S_m$, $m\leq n$, we say
%% that $\tau$ \emph{contains the pattern} $\pi$ if there exist indices
%% $1\leq i_1 < i_2 < \cdots < i_m\leq n$ such that, for all $1\leq
%% j,k\leq m$, $\tau(i_j) \leq \tau(i_k)$ if and only if $\pi(j) \leq
%% \pi(k)$.  If $\tau$ does not contain the pattern $\pi$, then we say
%% $\tau$ \emph{avoids} $\pi$.  Central to this paper will be
% pattern containment 
Given a word $w \in \wdsk{a}$ and a permutation $\pi\in S_m$, $m\leq
a$, we say that $w$ \emph{contains the pattern} $\pi$ if there
exist indices $1\leq i_1 < i_2 < \cdots < i_m\leq a$ such that, for
all $1\leq j,k\leq m$, $w(i_j) < w(i_k)$ if and only if $\pi(j) <
\pi(k)$ and $w(i_j) > w(i_k)$ if and only if $\pi(j) > \pi(k)$.  If
$w$ does not contain the pattern $\pi$, then we say $w$
\emph{avoids} $\pi$.  

We have defined pattern containment for words $w$ as that is the
generality in which we state Theorem~\ref{thm:cont}.  For pattern
avoidance, however, we will only need the case in which $w$ is itself
a permutation.  In particular, central to this paper will be
permutations that simultaneously avoid the patterns $3142$ and $2413$.
Being $3142,2413$-avoiding is one characterization of the class of
\emph{separable} permutations (see~\cite{seppaper}).  Throughout this
paper $\sigma$ will denote a separable permutation with
$\shape(\sigma) = \lambda = (\lambda_1,\lambda_2,\ldots)$.

% inversion poset and separable
Given a permutation $\pi\in S_n$, let $P(\pi)$ denote its
\emph{inversion poset}.  $P(\pi)$ has elements $(i,\pi(i))$ for $1\leq
i\leq n$ under the partial order $\prec$, in which $(a,b) \prec (c,d)$ if
and only if $a < c$ and $b < d$.  Note that increasing subsequences in
$\pi$ correspond to chains in $P(\pi)$. A longest increasing
subsequence of $\pi$ corresponds to a maximal chain in $P(\pi)$.

\begin{example}\label{ex:24133412}
The inversion poset of $2413$ is
$\displaystyle
 \begin{xy}
*!C\xybox{
\xymatrix@=2em{
(2,4) & (4,3) \\
(1,2) \ar@{-}[u] \ar@{-}[ur]& (3,1)  \ar@{-}[u]  } }
\end{xy}$
and that of $3142$ is 
$\displaystyle 
 \begin{xy}
*!C\xybox{
\xymatrix@=2em{
(3,4) & (4,2) \\
(1,3) \ar@{-}[u] & (2,1)  \ar@{-}[u]\ar@{-}[ul]  } }
\end{xy}$.
\end{example}

Example~\ref{ex:24133412} above immediately gives the following fact.

\begin{fact}\label{fact:sep}
  A permutation $\pi$ is separable if and only if its
  inversion poset $P(\pi)$ has no (induced) subposet isomorphic to
  $\displaystyle
  \begin{xy}
    *!C\xybox{
      \xymatrix@=1em{
        {*} & {*} \\
        {*} \ar@{-}[u] \ar@{-}[ur]& {*} \ar@{-}[u]  } }
  \end{xy}$.
\end{fact}

% Robinson-Schensted
We write our partitions with parts in decreasing order and make
no distinction between the positive and zero parts.  Given a partition
$\mu = (\mu_1, \mu_2, \ldots)$ of $n$ (denoted $\mu \vdash n$), the
associated \emph{Ferrers diagram} consists of $\mu_i$ left-justified
cells in the $i$-th row from the top.  A \emph{semistandard Young tableau
  of shape $\mu$} is a filling of the cells in this diagram with
positive integers such that the rows weakly increase from left to
right and the columns strictly increase from top to bottom.  The set
of such tableaux is denoted by $\ssyt(\mu)$.  A tableau
$T\in\ssyt(\mu)$, $\mu \vdash n$, is \emph{standard} if each number
from $1$ to $n$ appears in its filling.  The set of all such tableaux is
denoted by $\syt(\mu)$.  Given a semistandard tableau $T$, the
\emph{reading word} of $T$, $\rw(T)$, is the word obtained by reading
off the rows from left to right starting with the bottom row.  For
$\mu \vdash n$, define the \emph{superstandard tableau}
$T\in\syt(\mu)$ by filling in the rows from top to bottom.  That is,
by placing $1,2,\ldots,\mu_1$ in the first row,
$\mu_1+1,\mu_1+2,\ldots,\mu_1+\mu_2$ in the second row, etc.

The RSK correspondence yields a bijection between $\wdsk{a}$ and
$\cup_{\mu \vdash a} \ssyt(\mu)\times \syt(\mu)$~\cite{knuth}.  We
give a brief description of how to compute the pair $(P(w),Q(w))$ to
which a word $w\in\wdsk{a}$ corresponds.  Write $w=w'x$ with
$w'\in\wdsk{a-1}$.  By induction, we know that $w'$ maps to some pair
$(P(w'),Q(w'))$.  We \emph{row insert} $x$ in the first row of $P(w')$
as follows: If $x=x_1$ is greater than or equal to all elements in
this row, place $x_1$ at the end of the row.  Otherwise, find the
leftmost entry, $x_2$, in the row that is strictly greater than $x_1$.
Place $x_1$ in this position and ``bump'' $x_2$ to be inserted into
the next row.  This process generates some finite sequence
$x_1,\ldots,x_k$ of bumped elements and ends by adding $x_k$ at the
end of the $k$-th row, creating a new semistandard tableaux $P(w)$.
Set $Q(w)$ to have an $a$ in the new box (end of row $k$) created by
the bumping process.  The \emph{shape} of $w$, $\shape(w)$, is the
shape of $P(w)$ (or, equivalently, of $Q(w)$).

% containment, avoidance and supersequences.
% Robinson-Schensted
\begin{example}
  The permutation $\pi=7135264$ contains the pattern $4231$ but avoids
  $3412$.  Under the RSK correspondence, $w=2214312$ maps to
  $\left(P(w),Q(w)\right) = \left(
  \young(112,223,4)\,,\young(124,357,6)\right)$ with $\rw(P(w)) =
  4223112$.  Finally, the superstandard tableau of shape $(3,3,2)$ is
  $\young(123,456,78)$.
\end{example}

%%%%%%%%%%%%%%%%%%%%%%%%%%%%%%%%%%%%%%%%%%%%%%%%%%%%%%%%%%%%%%%%%%%%%
\section{Proof of Theorem~\ref{thm:cont}}
\label{sec:proof}

Many properties of a word $w$ translate to natural properties of the
associated tableaux.  For example, the length of the longest weakly
increasing subsequence of $w$ equals $\lambda_1$.  In fact, Greene's
Theorem~\cite{greene} gives a much more precise correspondence.

\begin{theorem}[Greene's Theorem]\label{thm:greene}
  Let $w$ be a word of shape $\lambda$.  For any $d\geq 0$ the sum
  $\lambda_1 + \cdots + \lambda_d$ equals the maximum number of
  elements in a disjoint union of $d$ weakly increasing subsequences
  of $w$.
\end{theorem}

In order to prove Theorem~\ref{thm:cont}, we will combine the insight
afforded by Greene's Theorem with the ability to exchange collections
of disjoint increasing subsequences with other collections for which
the number of intersections has, in a certain sense, been reduced.
Lemma~\ref{lem:N}, which is the only place separability explicitly
appears in our proof, allows us to perform these exchanges.

%%%%%%%%%%%%%%%%%%%%%%%%%%
%\begin{lemma}\label{lem:N}
%  Let $x=x_1x_2\cdots x_n$ be a shuffle of two disjoint increasing
%  words, $w$ and $w'$, and let $u=u_1u_2\cdots u_\ell$ be an increasing
%  subsequence of $x$.  Suppose $x$ avoids the patterns $3142$ and
%  $2413$.  Then there exist disjoint increasing subsequences $\alpha$
%  and $\beta$ such that $x=\alpha\cup\beta$ and $\alpha\cap
%  u=\emptyset$.
%\end{lemma}

\begin{lemma}\label{lem:N}
  Let $u$, $w$, and $w'$ be increasing subsequences of a
  separable permutation $\sigma$.  Assume further that $w$ and $w'$
  are disjoint.  Then there exist two disjoint increasing subsequences
  $\alpha$ and $\beta$, such that $\alpha \cup \beta = w \cup
  w'$ and $\alpha \cap u = \emptyset$.
\end{lemma}

\begin{proof}
  Write $u = u_0 \cup u_1$ with $u_0\cap (w \cup w') = \emptyset$ and
  $u_1 \subset w \cup w'$.  Since $\alpha \subset w \cup w'$,
  the requirement that $\alpha$ and $u_1$ be disjoint ensures that
  $\alpha$ and $u$ are disjoint as well.  Hence, without loss of
  generality, we may restrict our attention in the proof to the case
  in which $u\subset w \cup w'$.

  Let $x = w\cup w'$ be the subsequence of $\sigma$ containing both
  $w$ and $w'$.

  We give first a short proof by contradiction.  Consider the
  inversion poset $P(\sigma)$ of the separable permutation $\sigma$.
  Increasing subsequences are in correspondence with chains and we
  will regard them as such.  Assume there is no chain $\beta \subset
  (w \cup w')$, such that $u \subset \beta$ and $(w\cup w') \setminus
  \beta$ is also a chain.  Let $\gamma \subset (w\cup w')$ be some
  maximal chain, such that $u \subset \gamma$.  Then there exist two
  incomparable points $x,y \in (w \cup w') \setminus \gamma$, i.e. $x
  \not \gtrless y$ and thus belonging to the two different chains,
  e.g. $x \in w$, $y\in w'$.  By maximality, $x\cup \gamma$, $y \cup
  \gamma$ are not chains. Hence there exist $a,b \in \gamma$ for which
  $x \not \gtrless a$, $y \not \gtrless b$, so we must have $a \in w'$
  and $b \in w$. Assume $a \succ b$, then we must have $x \succ b$ and
  $y \prec a$.  We have $\displaystyle
     \begin{xy} 
     *!C\xybox{
       \xymatrix@=1em{
        {x} & {a} \\
        {b} \ar@{-}[u] \ar@{-}[ur]& {y} \ar@{-}[u]} }
  \end{xy}$ with $x\not \gtrless a, x \not \gtrless y, y\not \gtrless
  b$. This is a subposet of $P(\sigma)$ isomorphic to $\displaystyle
  \begin{xy}
    *!C\xybox{
      \xymatrix@=1em{
        {*} & {*} \\
        {*} \ar@{-}[u] \ar@{-}[ur]& {*} \ar@{-}[u]  } }
  \end{xy}$, contradicting Fact~\ref{fact:sep}.
   
  We now give a constructive proof, which allows us to find $\alpha$
  and $\beta$.  Let $z=w\cup w'$ be the subsequence of $\sigma$
  containing both $w$ and $w'$.  We can assume that $u$ is a
  subsequence of $z$.  First note that, since $z$ is a shuffle of two
  increasing disjoint words, $z$ also avoids the pattern $321$.  Also,
  since $u$ is a subsequence of $z$, there exist indices $i_1 < i_1 <
  \cdots < i_\ell$ such that $u_j = z_{i_j}$ for each $1\leq j\leq
  \ell$.  It will be convenient to augment our sequences by prepending
  a $u_0 = z_0 < \min\{z_i\}_{1\leq i \leq n}$ and appending a
  $u_{\ell+1} = z_{n+1} > \max\{z_i\}_{1\leq i\leq n}$.

  For each $1\leq j\leq \ell$, let $\beta^j$ be the sequence of
  left-to-right maxima from $z_{i_j}\cdots z_{i_{j+1}-1}$ whose
  values are greater than or equal to $u_j$ and less than $u_{j+1}$.
  Define $\beta^0$ analogously except with values \emph{greater than}
  $u_0$ and less than $u_1$.  Then, $\beta = \beta^0\cdots \beta^\ell$
  is, by construction, an increasing subsequence of $z$.  (Note that
  $\beta$ does not include $u_0$ or $u_{\ell+1}$.)

  We now need to show that $\alpha = z\setminus \beta$ is increasing.
  Suppose not.  Then there exists some $j$ such that $\alpha_j = z_a >
  \alpha_{j+1}$.  Let $m$ be the unique value such that $i_m < a
  < i_{m+1}$.  We now split into cases in order to argue that $z$ must
  contain one of the three patterns $321$, $3142$ or $2413$.

  \begin{enumerate}
  \item Suppose $\alpha_j > u_{m+1}$.  This implies $m<\ell$ (and
    hence that $u_{m+1}$ is an element of $z$).  We argue according to
    the region in which the point $\alpha_{j+1}$ lies (see Figure~\ref{fig:grid}).

  \begin{figure}[htbp]
    \centering
        {\scalebox{.4}{\includegraphics{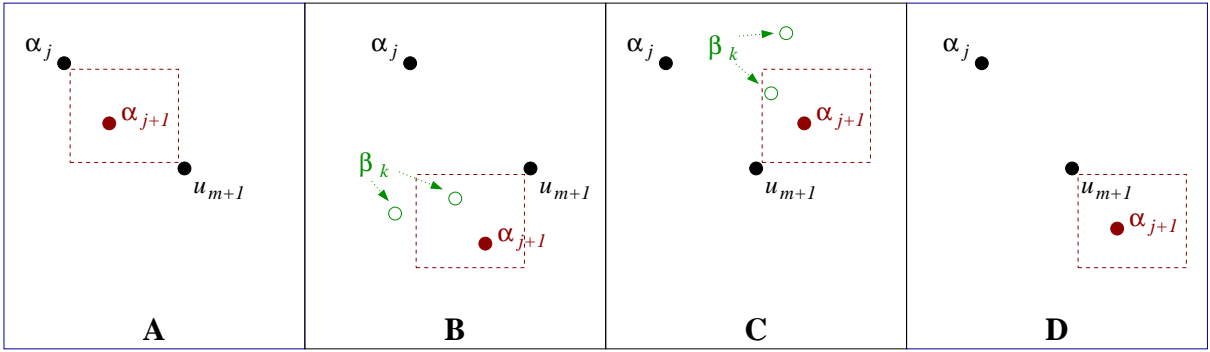}}}
        \caption{The cases in the proof of
          Lemma~\ref{lem:N} when $\alpha_j > u_{m+1}$.  Points are
          labeled by their $y$-values.}
        \label{fig:grid}
  \end{figure}

    \begin{enumerate}
    \item[A)] Then $\alpha_j\alpha_{j+1}u_{m+1}$ forms a $321$ pattern.
    \item[B)] Since $\alpha_{j+1}$ is not a left-to-right maximum,
      there must be some element $\beta_k$ lying to the northwest of
      $\alpha_{j+1}$ yet below $u_{m+1}$.  If $\beta_k$ lies to the
      left of $\alpha_j$, then $\beta_k\alpha_j\alpha_{j+1}u_{m+1}$
      forms a $2413$ pattern.  Otherwise,
      $\alpha_j\beta_k\alpha_{j+1}$ forms a $321$ pattern.
    \item[C)] As above, there must be some element $\beta_k$ lying to
      the northwest of $\alpha_{j+1}$ yet to the right of $u_{m+1}$.
      If $\beta_k$ lies above $\alpha_j$, then
      $\alpha_ju_{m+1}\beta_k\alpha_{j+1}$ forms a $3142$ pattern.
      Otherwise, $\alpha_j\beta_k\alpha_{j+1}$ forms a $321$ pattern.
    \item[D)] Then $\alpha_ju_{m+1}\alpha_{j+1}$ forms a $321$ pattern.
    \end{enumerate}

  \item Suppose $\alpha_j < u_{m+1}$.  Since $\alpha_j$ is not a left-to-right maximum, there must be some element
    $\beta_k$ (possibly $u_m$) lying to the northwest of $\alpha_j$.  Hence
    $\beta_k\alpha_j\alpha_{j+1}$ forms a $321$ pattern.
  \end{enumerate}
\end{proof}

\begin{example}
  Figure~\ref{fig:twoseqs} illustrates the sequences $\alpha$ and
  $\beta$ that arise from the construction of Lemma~\ref{lem:N}.  The
  two original sequences shuffled together are connected by dotted
  lines.  The elements of $u$ are illustrated by open circles.  The
  boxes indicate the regions in which the elements of $\beta$ (other
  than those of $u$ itself) are required to lie.  Finally, the
  sequence $\beta$ is connected by the thick, dashed line.
  \begin{figure}[htbp]
    \centering
        {\scalebox{.4}{\includegraphics{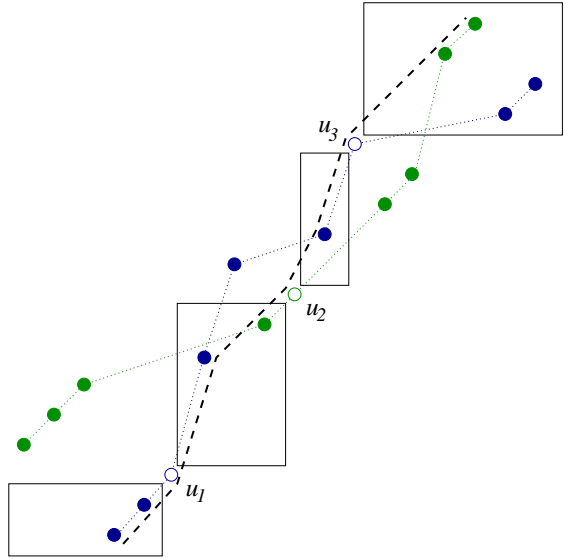}}}
        \caption{Example application of Lemma~\ref{lem:N}.}
        \label{fig:twoseqs}
  \end{figure}
\end{example}

%%%%%%%%%%%%%%%%%%%%%%%%%%
\begin{prop}\label{prop:ext}
  Let $k\geq 0$ and $s^1,\ldots,s^k$ be disjoint (possibly empty)
  increasing subsequences of the separable permutation $\sigma$.  Then
  there exists an increasing subsequence $u^{k+1}$, disjoint from each
  $s^i$, such that $|u^{k+1}| \geq \lambda_{k+1}$.
\end{prop}
\begin{proof}
  Let $V = \{v^1,\ldots,v^{k+1}\}$ be any collection of $k+1$
  disjoint, increasing subsequences of $\sigma$ of maximum total
  length.  Let $m=m(V)$ be the maximum index such that $v^j$ is
  disjoint from $s^i$ for all $1\leq i< m \leq j$.  We prove by
  induction that $m$ can be taken to be $k+1$.
  
  So, assume $1 \leq m < k+1$.  At least one $v^j$ with $j > m$ must
  intersect $s^m$, otherwise $m$ would be larger than it is.  But
  repeated application of Lemma~\ref{lem:N} to the elements of
  $\{v^m,\ldots,v^{k+1}\}$ produces a new set $V' =
  \{v^1,\ldots,v^{m-1},\tilde{v}^{m},\ldots,\tilde{v}^{k+1}\}$ in
  which $j=m$ is the only value for which $\tilde{v}^j$ intersects
  $s^m$.  Hence $m(V') \geq m(V)+1$, completing the induction step.

  By the preceding paragraph, we may assume that $v^{k+1}$ is disjoint
  from all of the $s^i$.  Since the elements of $V$ are of maximum
  total length, $|v^1| + \cdots + |v^{k+1}| = \lambda_1 + \cdots +
  \lambda_{k+1}$.  If $|v^{k+1}|$ were less than $\lambda_{k+1}$, then
  $v^1,\ldots,v^k$ would have total length greater than $\lambda_1 +
  \cdots + \lambda_k$.  This is impossible.  Hence $|v^{k+1}| \geq
  \lambda_{k+1}$ as desired.
\end{proof}
\begin{example}
  Consider $x=10652438ba97$ (where we use $a$ for $10$ and $b$ for
  $11$).  The shape of $x$ is $(5,3,2,2)$.  Suppose we have
  $u^1=0248b$, $u^2=167$ and $u^3=5a$ and wish to find a disjoint
  increasing subsequence $u^4$ of length $2$.  We could, of course,
  simply use the remaining two elements, $3$ and $9$.  However, in
  order to illustrate the proofs of Proposition~\ref{prop:ext} and
  Theorem~\ref{prop:sep}, we show how to generate this sequence from an
  arbitrarily chosen $4$-tuple of disjoint increasing subsequences of
  maximum total length: $V = \{68b,049,237,15a\}$.

  So, set $k=3$ and $s^j=u^j$ for $j\in\{1,2,3\}$.  Consider the
  argument of Proposition~\ref{prop:ext}.  $m=m(V)=1$.  Let $u=u^1$,
  $w=68b$ and $w'=049$.  Lemma~\ref{lem:N} yields $\alpha = 69$ and
  $\beta=048b$.  Applying the lemma again with $w=048b$ and $w'=237$
  yields $\alpha=37$ and $\beta=0248b$.  This produces the new
  $4$-tuple $V'=\{0248b,69,37,15a\}$ with $m(V')=2$.

  Now set $u=u^2$.  Once again, an application of the lemma with
  $w=69$ and $w'=15a$ yields $\alpha=59$ and $\beta=16a$, while a
  following application to $w=16a$ and $w'=37$ yields $\alpha=3a$ and
  $\beta=167$.  This produces the new $4$-tuple
  $V''=\{0248b,167,59,3a\}$ with $m(V'') = 3$.  

  A final application of the lemma with $u=u^3=5a$, $w=59$ and $w'=3a$ yields
  the sought for $u^4=39$.
\end{example}

%%%%%%%%%%%%%%%%%%%%%%%%%%%%%%%%%%%%%%%%%%%%%%%%%%%%%%%%%%%%%%%%%%%%%%%%%
% This seems kind of clunky.
% Need to be careful about words containing patterns.  Does 11 contain
% the pattern 12?
\begin{proof}[Proof of Theorem~\ref{thm:cont}]
  Let $\shape(w) = \mu = (\mu_1,\mu_2,\ldots)$.  Let $\sigma'$ be any
  subsequence of $w$ in the same relative order as the elements of
  $\sigma$; i.e., $w$ contains $\sigma$ at the positions of $\sigma'$.
  By Greene's Theorem applied to $w$, for any $k \geq 1$ there exist
  $k$ disjoint increasing subsequences $w^1,\ldots,w^k$ with
  $|w^1|+\cdots+|w^k| = \mu_1+\cdots+\mu_k$.  The intersection
  $\sigma' \cap w^i$ induces a subsequence of $\sigma$ we denote by
  $s^i$.  These $s^i$ are then $k$ disjoint increasing subsequences of
  $\sigma$.  By Proposition~\ref{prop:ext}, there is an increasing
  subsequence $u$ of $\sigma$, disjoint from the $s^i$s, with length
  at least $\lambda_{k+1}$.  The mapping $\sigma \mapsto \sigma'$
  induces a corresponding map of $u$ to a subsequence $u'$ of $w$.  It
  follows then that $u'$ is disjoint from each $w^i$ as well.  Then
  $w^1,\ldots,w^k,u'$ are $k+1$ disjoint increasing subsequences in
  $w$.  By Greene's Theorem,
  \begin{equation*}
    |w^1| + \cdots + |w^k| + |u'| \leq \mu_1 + \cdots + \mu_k + \mu_{k+1}.
  \end{equation*}
  Hence $|u'| \leq \mu_{k+1}$.  We also know by construction that
  $\lambda_{k+1}\leq |u| = |u'|$.  Combining these equalities and
  running over all $k$ yields $\lambda \subseteq \mu$ as desired.
\end{proof}

%%%%%%%%%%%%%%%%%%%%%%%%%%%%%%%%%%%%%%%%%%%%%%%%%%%%%%%%%%%%%%%%%%%%
\subsection{Relationship to Greene's Theorem}
\label{sec:gt}

Greene's Theorem only tells us about the maximum \emph{sum} of lengths
of disjoint increasing sequences.  It is \emph{not} generally true
that one can find $d$ disjoint increasing subsequences
$u^1,u^2,\ldots,u^d$ of $w$ with $u^i$ of length $\lambda_i$ for each
$i$.  In other words, the shape of a word does not tell you the
lengths of the subsequences in a set of $d$ disjoint increasing
subsequences of maximum total length; it just tells you the maximum
total length.
% actually it tells you a little more in terms of upper bounds for
% j-sets with j < k.

\begin{example}\label{ex:nonsep}
  Consider the permutation $w = 236145$ of shape $(4,2)$.  The
  only increasing subsequence of length four is $2345$.  However, the
  remaining two entries appear in decreasing order.  Greene's Theorem
  tells us that we should be able to find two disjoint increasing
  subsequences of total length $6$.  Indeed, $236$ and $145$ work.
\end{example}

Nonetheless, such a collection of subsequences $\{u^i\}$ does exist
when $\sigma$ is a separable permutation.

\begin{prop}\label{prop:sep}
  Let $\sigma$ be a separable permutation of shape $\lambda$.  For any
  $d\geq 1$, there exist $d$ disjoint, increasing subsequences
  $u^1,\ldots,u^d$ such that the length of each $u^i$ is given by
  $\lambda_i$.
\end{prop}

Theorem~\ref{thm:cont} and Proposition~\ref{prop:sep} are
superficially similar.  We have already shown how
Theorem~\ref{thm:cont} follows from Proposition~\ref{prop:ext} (and
Greene's Theorem).  Proposition~\ref{prop:sep} follows even more
immediately.

\begin{proof}[Proof of Proposition~\ref{prop:sep}]
  We can construct such a sequence via $d$ applications of
  Proposition~\ref{prop:ext}.  In particular, given the
  $u^1,\ldots,u^i$ for some $0\leq i< d$, produce $u^{i+1}$ by
  applying the proposition with $k=i$ and $s^j = u^j$ for $1\leq j\leq
  k$.  
\end{proof}

As pointed out to us by a referee to an earlier version of this paper,
Proposition~\ref{prop:sep} has a very simple proof relying on the
recursive definition of a separable permutation as one that can be
built up by direct and skew sums (see~\cite{seppaper}).  However, we
have been unable to follow a correspondingly direct proof of
Theorem~\ref{thm:cont}.

%%%%%%%%%%%%%%%%%%%%%%%%%%%%%%%%%%%%%%%%%%%%%%%%%%%%%%%%%%%%%%%%%%%%%%
\section{Supersequences}
\label{sec:super}

Let $B \subseteq S_n$ be a set of permutations.  A word $w$ is a
\emph{supersequence of $B$} if, for all $\sigma\in B$, $\sigma$ is a
subsequence of $w$.  Note that for $w$ to be a supersequence of
$\{\sigma\}$, the actual entries of $\sigma$ must occur (in the same
order) in $w$; this is in contrast to pattern containment in which we
need only find elements of $w$ in the same relative order.

\begin{example}
  The word $w = 2214312$ is a supersequence of $132$ but not of $321$.
  In fact, $w$ is a supersequence of the set $B = \{132,312,213\}$.
\end{example}

% talk about evidence that this upper bound is not tight
Let $\scs{n}(B)$ denote the minimum length of a supersequence of the
set $B$.  An upper bound of $\scs{n}(S_n) \leq n^2-2n+4$ has been
proven by a number of different researchers in various contexts and
generalities.  See in
particular~\cite{adleman,galbiati,Koutas,mohanty,newey,savage}.
Recently, an upper bound of $n^2-2n+3$ was proven constructively for
$n \geq 10$ by Z{\u a}linescu~\cite{zalinescu}.
% This bound was conjectured to be tight.  However, a forum post by
% \textsc{dirtcheap}~\cite{dirtcheap} shows this is not the case: For
% $n=10$ (using the alphabet $\{0,1,\ldots,9\}$ for clarity) take
% % checked on computer, 02.08.10
% \begin{multline}
%   w = 90123\ 45678\ 90123\ 45679\ 80123\ 45697\ 01234\ 85967\ 01234\ 95867\\
%   01239\ 45670\ 12893\ 45670\ 12983\ 45670\ 129.
% \end{multline}
% %\begin{multline}
% %  w = 9\ 0123456789\ 012345679\ 8\ 0123456\ 9\ 7\ 01234\ 8\ 5\ 9\ 67\ 01234\ 9\ 5\ 8\ 67\ 0123\ 9\\
% %  \ 4567\ 012\ 89\ 34567\ 012\ 9\ 8\ 34567\ 012\ 9.
% %\end{multline}
% The length of $w$ is only $83 < 84 = 10^2-2\cdot 10 + 4$.  In
% addition, a counterexample for $n=11$ of length $102 < 103 =
% 11^2-2\cdot 11 + 4$, is given by
% \textsc{tokenmathguy}~\cite{tokenmathguy}.  
Kleitman and Kwiatkowski~\cite{Kleitman} have shown that $\scs{n}(S_n)
\geq n^2 - Cn^{7/4+\varepsilon}$ where $\varepsilon > 0$ and $C$
depends on $\varepsilon$.

For certain sets $B$, we can construct a lower bound for
$\scssub{n}{B}$ by considering the union of $\shape(\sigma)$ as
$\sigma$ runs over the elements of $B$.

\begin{lemma}
  If $T$ is the superstandard tableau of shape $\lambda$, then
  $\rw(T)$ is a $2413$,$3142$-avoiding permutation (i.e., is separable).
\end{lemma}
\begin{proof}
  In fact, $\rw(T)$ avoids the pattern $213$: For $i < j$, the entries
  in row $j$ are greater than, and precede, the entries in row $i$.
  Hence $\rw(T)$ avoids both $2413$ and $3142$.
\end{proof}

Fix $k > 0$ and $B = \{\sigma_1,\ldots,\sigma_k\}$ with each
$\sigma_i$ separable.  It follows then from Theorem~\ref{prop:sep} that
for any supersequence $w$ of $B$, $\shape(w) \supseteq \cup_i
\shape(\sigma_i)$.  Hence, if we choose the $\sigma_i$ so that the
Ferrers diagrams of shapes $\shape(\sigma_i)$ overlap as little as
possible, we force any supersequence $w$ to be relatively long.

\begin{example}
  Let $n = 9$ and $k = 5$.  Choose the permutations
  $B = \{\sigma_1,\ldots,\sigma_5\}$ as
  \begin{align*}
    \sigma_1 &= 123456789, \quad \shape(\sigma_1) = (9),\\
    \sigma_2 &= 678912345, \quad \shape(\sigma_2) = (5,4),\\
    \sigma_3 &= 789456123, \quad \shape(\sigma_3) = (3,3,3),\\
    \sigma_4 &= 978563412, \quad \shape(\sigma_4) = (2,2,2,2,1),\\
    \sigma_5 &= 987654321, \quad \shape(\sigma_5) = (1,1,1,1,1,1,1,1,1).
  \end{align*}
  The union of the corresponding Ferrers diagrams is $\Yboxdim6pt
  \yng(9,4,3,2,1,1,1,1,1)$; we see that
  $|\cup_{i=1}^5 \shape(\sigma_i)| = 23$.  A computer search
  provides the length-$23$ supersequence
  $$6\,9\,7\,8\,7\,5\,9\,6\,5\,4\,3\,1\,2\,3\,4\,5\,6\,7\,8\,9\,1\,2\,3,$$
  thereby showing that this bound is optimal.
\end{example}

Let $\mu(n)$ be the Ferrers diagram obtained by taking the union of
all Ferrers diagrams of size $n$.

\begin{prop}
  Let $\tau(i)$ denote the number of divisors of $i$.  Then
  $|\mu(n)| = \sum_{i=1}^n \tau(i)$ and the number of corners (i.e.,
  distinct row lengths of $\mu(n)$) is given by $\lfloor
  \sqrt{4n+1}\rfloor-1$.
  % Another expression for |mu(n)| is apparently \sum_{i=1}^n \lfloor n/i\rfloor.
\end{prop}
\begin{proof}
  For each divisor $d$ of $n$, the shape
  $(\overbrace{n/d,\ldots,n/d}^{d\text{ times}})$ will be contained in
  $\mu(n)$.  Furthermore, each cell $(d,n/d)$ will be a corner that is
  not part of $\mu(n-1)$.  The result $|\mu(n)| = \sum_{i=1}^n
  \tau(i)$ then follows by induction.  (In fact, the nested sequence
  of Ferrers diagrams $\mu(1) \subset \mu(2)\subset \cdots \subset
  \mu(n)$ can be thought of as a semistandard Young tableau of shape
  $\mu(n)$ in which the label $i$ occurs $\tau(i)$ times.)

  We now prove that the number of corners of $|\mu(n)|$ is given by
  $\lfloor \sqrt{4n+1}\rfloor -1$.  Let $k$ be the largest integer for
  which a $k\times k$ square is contained in the diagram of $\mu(n)$,
  that is, $k$ is the number of cells on the main diagonal in
  $\mu(n)$. We have that $k^2\leq n$.  The cell $(k,k)$ is a corner
  of $\mu(n)$ if and only if $k(k+1)>n$,
  i.e. $(\underbrace{k,\ldots,k}_{k+1})$ is not contained in any
  diagram of size $n$. We claim that the rows $1,\ldots,k$ of $\mu(n)$
  will each contain a corner of $\mu(n)$.  This is trivially true for
  the first row.  For $1<i\leq k$, row $i$ ends at $(i,\lfloor
  n/i\rfloor)$ while the row above ends at $(i-1,\lfloor
  n/(i-1)\rfloor)$.  Since $k\leq \sqrt{n}$, $\lfloor n/i\rfloor <
  \lfloor n/(i-1)\rfloor$ and we have a corner in the $i$-th row as
  claimed.
% it follows from the facts that $k\leq \sqrt{n}$ along 
%  since for $i \leq k\leq \sqrt{n}$
%   row $i$ ends at $(i,\lfloor n/i \rfloor)$ and $\lfloor n/i \rfloor <
%   \lfloor n/(i-1) \rfloor$. 
  The same argument holds for the first $k$ columns, so the total
  number of corners is $2k-1$ if $(k,k)$ is a corner and $2k$
  otherwise. We have that $(2k+1)^2 \geq 4n+1 \geq 4k^2+1$, with the
  first inequality strict if and only if $(k,k)$ is a corner. Hence
  the number of corners is indeed $\lfloor\sqrt{4n+1}\rfloor -1$.
%%  Let $\mu(n) = (\mu_1,\mu_2,\ldots,\mu_n)$.  For each divisor $d$ of
%%  $n$, the shape $(\overbrace{n/d,\ldots,n/d}^{d\text{ times}})$ will be
%%  contained in $\mu(n)$.  Furthermore, each cell $(d,n/d)$ will be a
%%  corner.  The result $|\mu(n)| = \sum_{i=1}^n d(i)$ then follows by
%%  induction.
%%
%%  Let $k = \lfloor \sqrt{n}\rfloor$.  Then $n$ can be written uniquely
%%  as $k^2+a$ for $0\leq a \leq 2k$.  The number of corners in $\mu(n)$
%%  is $2k-1$ if $0\leq a\leq k-1$ and $2k$ if $k\leq a\leq 2k$.  A
%%  direct computation shows that $\lfloor \sqrt{4n+1}\rfloor-1$ reduces
%%  to these formulas according to the value of $a$.
\end{proof}

It is a standard fact that $\sum_{i=1}^n \tau(i) \sim n(\ln n + 2\gamma +
\cdots)$ (see, e.g.,~\cite[Theorem 3.3]{apostol}).  Hence, for any $n$
we can find $\lfloor \sqrt{4n+1}\rfloor-1$ permutations whose
supersequence is of length at least $n(\ln n + 2\gamma+\cdots)$.
Compare this with $n!$ permutations having a supersequence of length
$O(n^2)$.

\section{Acknowledgments}
The authors would like to thank Dan Archdeacon, Jeff Buzas and Jeff
Dinitz for useful conversations leading to this paper.

%%%%%%%%%%%%%%%%%%%%%%%%%%%%%%%%%%%%%%%%%%%%%%%%%%%%%%%%%%%%%%%%%%%%%%
% \bibliographystyle{amsplain}
% \bibliography{scs}

\providecommand{\bysame}{\leavevmode\hbox to3em{\hrulefill}\thinspace}
\providecommand{\MR}{\relax\ifhmode\unskip\space\fi MR }
\providecommand{\MRhref}[2]{%
  \href{http://www.ams.org/mathscinet-getitem?mr=#1}{#2}
}
\providecommand{\href}[2]{#2}

\end{document}